\documentclass[12pt]{article}

\usepackage{amsmath, epsfig, cite, setspace}
\usepackage{amssymb}
\usepackage{amsfonts}
\usepackage{latexsym}
\usepackage{amsthm}

\usepackage{soul}

\makeatletter
\renewcommand{\@seccntformat}[1]{{\csname the#1\endcsname}.\hspace{.5em}}
\makeatother

\newtheorem{thm}{Theorem}[section]

\newtheorem{conj}[thm]{Conjecture}
\newtheorem{lem}[thm]{Lemma}

\renewcommand{\thefootnote}{*}

\numberwithin{equation}{section}

\begin{document}

\begin{center}
{\large\bf $q$-Analogues of Dwork-type supercongruences}
\end{center}

\vskip 2mm \centerline{Victor J. W. Guo}
\begin{center}
{\footnotesize School of Mathematics and Statistics, Huaiyin Normal
University, Huai'an 223300, Jiangsu, People's Republic of China\\
{\tt jwguo@hytc.edu.cn }  }
\end{center}


\vskip 0.7cm \noindent{\bf Abstract.} In 1997, Van Hamme conjectured 13 Ramanujan-type supercongruences. All of the 13 supercongruences
have been confirmed by using a wide range of methods. In 2015, Swisher
conjectured Dwork-type supercongruences  related to the first 12 supercongruences of Van Hamme. Here we prove that the (C.3) and (J.3) supercongruences of  Swisher
are true modulo $p^{3r}$ (the original modulus is $p^{4r}$) by establishing $q$-analogues of them. Our proof will use
the ``creative microscoping" method, recently introduced by the author in collaboration with Zudilin. We also raise conjectures on
$q$-analogues of an equivalent form of the (M.2) supercongruence of Van Hamme, partially answering a question at the end of [Adv. Math. 346 (2019), 329--358].


\renewcommand{\thefootnote}{**}

\section{Introduction}
In 1914, Ramanujan \cite{Ramanujan} mysteriously stated 17 hypergeometric series representations of $1/\pi$, including
\begin{equation*}
\sum_{k=0}^\infty (6k+1)\frac{(\frac{1}{2})_k^3}{k!^3 4^k}
=\frac{4}{\pi}, 
\end{equation*}
where $(a)_n=a(a+1)\cdots(a+n-1)$ denotes the Pochhammer symbol.
In 1997, Van Hamme \cite{Hamme} listed 13 interesting $p$-adic
analogues of Ramanujan's and Ramanujan-type formulas, such as
\begin{align}
\sum_{k=0}^{(p-1)/2} (4k+1)\frac{(\frac{1}{2})_k^4}{k!^4}
&\equiv p\pmod{p^3},  \label{eq:c2} \\[5pt]
\sum_{k=0}^{(p-1)/2} (6k+1)\frac{(\frac{1}{2})_k^3}{k!^3 4^k}
&\equiv (-1)^{(p-1)/2}p\pmod{p^4},  \label{eq:j2}
\end{align}
where $p>3$ is a prime.
The supercongruence \eqref{eq:c2} was proved by Van Hamme \cite{Hamme} himself and Long \cite{Long}
proved that it also holds modulo $p^4$ by using hypergeometric series
identities and evaluations. The supercongruence \eqref{eq:j2} was also confirmed by Long \cite{Long}.
It was not until 2016 that the last supercongruence of Van Hamme was proved by Osburn and Zudilin \cite{OZ}.
For some background on Ramanujan-type supercongruences,
we refer the reader to Zudilin's paper \cite{Zud2009}.

In 2015, Swisher \cite{Swisher} proved and reproved several supercongruences of Van Hamme by utilizing Long's method.
Furthermore, she proposed some conjectures on supercongruences that generalize Van Hamme's supercongruences (A.2)--(L.2). In
particular, Swisher's conjectural (C.3) and (J.3) supercongruences, which are generalizations of \eqref{eq:c2} and \eqref{eq:j2} respectively,
can be stated as follows: for any prime $p>3$,
\begin{align}
\sum_{k=0}^{(p^r-1)/2} (4k+1)\frac{(\frac{1}{2})_k^4}{k!^4}
&\equiv p \sum_{k=0}^{(p^{r-1}-1)/2} (4k+1)\frac{(\frac{1}{2})_k^4}{k!^4} \pmod{p^{4r}},  \label{eq:c3}\\[5pt]
\sum_{k=0}^{(p^r-1)/2} (6k+1)\frac{(\frac{1}{2})_k^3}{k!^3 4^k}
&\equiv (-1)^{(p-1)/2}p \sum_{k=0}^{(p^{r-1}-1)/2} (6k+1)\frac{(\frac{1}{2})_k^3}{k!^3 4^k} \pmod{p^{4r}}.  \label{eq:j3}
\end{align}

In recent years, $q$-analogues of congruences and supercongruences
have been investigated by many authors (see, for example,
\cite{Gorodetsky,Guo2018,Guo-J,Guo2,Guo5,Guo-m3,GL18,GPZ,GS,GS2,GS3,GW,GuoZu,GuoZu2,NP,Straub,Tauraso2}).
In particular, using the $q$-WZ method \cite{WZ} the author and Wang \cite{GW} gave a $q$-analogue of \eqref{eq:c2}: for odd $n$,
\begin{align}
\sum_{k=0}^{(n-1)/2}[4k+1]\frac{(q;q^2)_k^4}{(q^2;q^2)_k^4}
\equiv q^{(1-n)/2}[n]+\frac{(n^2-1)(1-q)^2}{24}\,q^{(1-n)/2}[n]^3 \pmod{[n]\Phi_n(q)^3}.  \label{eq:gw}
\end{align}
Moreover, the author \cite{Guo-J} conjectured that, for odd $n$,
\begin{align}
&\sum_{k=0}^{(n-1)/2}q^{k^2}[6k+1]\frac{(q;q^2)_k^2 (q^2;q^4)_k }{(q^4;q^4)_k^3} \notag\\[5pt]
&\quad\equiv (-q)^{(1-n)/2}[n]+\frac{(n^2-1)(1-q)^2}{24}\,(-q)^{(1-n)/2}[n]^3 \pmod{[n]\Phi_n(q)^3}.  \label{eq:qj2}
\end{align}
Here and in what follows we adopt the standard $q$-notation:
$(a;q)_n=(1-a)(1-aq)\cdots (1-aq^{n-1})$ is the {\em $q$-shifted factorial};
$[n]=[n]_q=1+q+\cdots+q^{n-1}$ is the {\em $q$-integer};
and $\Phi_n(q)$ stands for the $n$-th {\em cyclotomic polynomial} in $q$:
\begin{align*}
\Phi_n(q)=\prod_{\substack{1\leqslant k\leqslant n\\ \gcd(n,k)=1}}(q-\zeta^k),
\end{align*}
where $\zeta$ is an $n$-th primitive root of unity.
The author \cite{Guo-J} himself proved that the $q$-supercongruence \eqref{eq:qj2} is true modulo $[n]\Phi_n(q)$.
Shortly afterwards, the author and Zudilin \cite{GuoZu} proved that \eqref{eq:qj2} holds
modulo $[n]\Phi_n(q)^2$ by a newly introduced method of creative microscoping.

It is worth mentioning that \eqref{eq:c3} and \eqref{eq:j3} have the following companions: for any prime $p>3$,
\begin{align}
\sum_{k=0}^{p^r-1} (4k+1)\frac{(\frac{1}{2})_k^4}{k!^4}
&\equiv p \sum_{k=0}^{p^{r-1}-1} (4k+1)\frac{(\frac{1}{2})_k^4}{k!^4} \pmod{p^{4r}}, \label{eq:cc}\\[5pt]
\sum_{k=0}^{p^r-1} (6k+1)\frac{(\frac{1}{2})_k^3}{k!^3 4^k}
&\equiv (-1)^{(p-1)/2}p \sum_{k=0}^{p^{r-1}-1} (6k+1)\frac{(\frac{1}{2})_k^3}{k!^3 4^k} \pmod{p^{4r}}, \label{eq:jj}
\end{align}
which were conjectured in \cite{GW} and \cite{Guo-J}, respectively.

Given a prime $p$, a power series $f(z)=\sum_{k=0}^\infty A_k z^k$ is said to satisfy the Dwork congruence \cite{Dwork,MV}
if
\begin{equation}
\frac{f_{r+1}(z)}{f_r(z^p)}\equiv \frac{f_{r}(z)}{f_{r-1}(z^p)}\pmod{p^r\mathbb{Z}_p[[z]]}\quad\text{for $r=1,2,\ldots$},
\label{eq:condition}
\end{equation}
where
$$
f_r(z)=\sum_{k=0}^{p^r-1}A_k z^k
$$
denotes the truncation of $f(z)$.
Moreover, if the modulus in \eqref{eq:condition} can be replaced by $p^s\mathbb{Z}_p[[z]]$ for $s=s_r>r$, then we say that it satisfies
a Dwork supercongruence. Formally, we need the condition $f_1(z^p)=\sum_{k=0}^{p-1}A_k z^{pk}\not\equiv 0\pmod{p\mathbb{Z}_p[[z]]}$ satisfied
to make sense of \eqref{eq:condition}. However this may be relaxed to $f_1(z^p)\not\equiv 0\pmod{p^m\mathbb{Z}_p[[z]]}$
if the congruences \eqref{eq:condition} hold modulo $p^{mr}\mathbb{Z}_p[[z]]$ for some $m>1$.
This allows one to view Swisher's conjectures from \cite{Swisher} as particular instances of Dwork-type supercongruences.

The aim of this paper is to establish (partial) $q$-analogues of \eqref{eq:c3}, \eqref{eq:j3}, \eqref{eq:cc} and \eqref{eq:jj}.
We divide them into two theorems. It is reasonable to call the $q$-congruences in the theorems
{\em $q$-analogues of Dwork-type supercongruences}.
\begin{thm}\label{thm:main}
Let $n>1$ be an odd integer and let $r\geqslant 1$. Then, modulo $[n^r]\prod_{j=1}^r\Phi_{n^j}(q)^2$,
\begin{align}
\sum_{k=0}^{(n^r-1)/2}[4k+1]\frac{(q;q^2)_k^4} {(q^2;q^2)_k^4}
&\equiv q^{(1-n)/2} [n] \sum_{k=0}^{(n^{r-1}-1)/2}[4k+1]_{q^n}\frac{(q^{n};q^{2n})_k^4} {(q^{2n};q^{2n})_k^4}, \label{q-c3-1}\\[5pt]
\sum_{k=0}^{n^r-1}[4k+1]\frac{(q;q^2)_k^4} {(q^2;q^2)_k^4}
&\equiv q^{(1-n)/2} [n] \sum_{k=0}^{n^{r-1}-1}[4k+1]_{q^n}\frac{(q^{n};q^{2n})_k^4} {(q^{2n};q^{2n})_k^4}.  \label{q-c3-2}
\end{align}
\end{thm}

It is easy to see that, when $n=p$ is a prime and $q\to 1$, the $q$-supercongruences \eqref{q-c3-1} and \eqref{q-c3-2} reduce
to \eqref{eq:c3} and \eqref{eq:cc} modulo $p^{3r}$.

\begin{thm}\label{thm:main-2}
Let $n>1$ be an odd integer and let $r\geqslant 1$. Then, modulo $[n^r]\prod_{j=1}^r\Phi_{n^j}(q)^2$,
\begin{align}
\sum_{k=0}^{(n^r-1)/2} q^{k^2}[6k+1]\frac{(q;q^2)_k^2 (q^2;q^4)_k }{(q^4;q^4)_k^3}
&\equiv (-q)^{(1-n)/2}[n]\sum_{k=0}^{(n^{r-1}-1)/2} q^{nk^2}[6k+1]_{q^n}\frac{(q^n;q^{2n})_k^2 (q^{2n};q^{4n})_k }{(q^{4n};q^{4n})_k^3}, \label{q-j3-1}\\[5pt]
\sum_{k=0}^{n^r-1} q^{k^2}[6k+1]\frac{(q;q^2)_k^2 (q^2;q^4)_k }{(q^4;q^4)_k^3}
&\equiv (-q)^{(1-n)/2}[n]\sum_{k=0}^{n^{r-1}-1} q^{nk^2}[6k+1]_{q^n}\frac{(q^n;q^{2n})_k^2 (q^{2n};q^{4n})_k }{(q^{4n};q^{4n})_k^3}.  \label{q-j3-2}
\end{align}
\end{thm}

Similarly, when $n=p$ is a prime and $q\to 1$, the $q$-supercongruences \eqref{q-j3-1} and \eqref{q-j3-2} reduce
to \eqref{eq:j3} and \eqref{eq:jj} modulo $p^{3r}$.

We prove Theorems \ref{thm:main} and \ref{thm:main-2} in Sections 2 and 3, respectively. We shall accomplish this by using
the creative microscoping method from \cite{GuoZu}. More precisely, we shall give parametric generalizations of
Theorems \ref{thm:main} and \ref{thm:main-2}. In Section 4, we  propose several conjectures
on $q$-analogues of an equivalent form of the (M.2) supercongruence of Van Hamme \cite{Hamme}, thus answering in part a suspicion of the author and Zudilin
in \cite{GuoZu}.

\section{Proof of Theorem \ref{thm:main}}
We first give the following result, which follows from  the $c=1$ case of \cite[Theorem 4.2]{GuoZu}.
\begin{lem}\label{lem:1}
Let $n$ be a positive odd integer. Then
\begin{align*}
\sum_{k=0}^{(n-1)/2}[4k+1]\frac{(aq;q^2)_k (q/a;q^2)_k(q;q^2)_k^2} {(aq^2;q^2)_k (q^2/a;q^2)_k(q^2;q^2)_k^2}
\equiv 0\pmod{[n]}.
\end{align*}
\end{lem}

We also need the following lemma, which is a also special case of \cite[Theorem 4.2]{GuoZu}. For the reader's convenience,
we give a two-line proof here.
\begin{lem}\label{lem:2}
Let $n$ be a positive odd integer. Then
\begin{align*}
\sum_{k=0}^{(n-1)/2}[4k+1]\frac{(q^{1-n};q^2)_k (q^{1+n};q^2)_k(q;q^2)_k^2} {(q^{2-n};q^2)_k (q^{2+n};q^2)_k(q^2;q^2)_k^2}
=q^{(1-n)/2}[n].
\end{align*}
\end{lem}
\begin{proof}
Recall that Jackson's ${}_6\phi_5$ summation formula can be written as
\begin{equation*}
\sum_{k=0}^N\frac{(1-aq^{2k})(a;q)_k(b;q)_k(c;q)_k(q^{-N};q)_k}
{(1-a)(q;q)_k(aq/b;q)_k(aq/c;q)_k(aq^{N+1};q)_k}\biggl(\frac{aq^{N+1}}{bc}\biggr)^k
=\frac{(aq;q)_N (aq/bc;q)_N}{(aq/b;q)_N(aq/c;q)_N}
\end{equation*}
(see \cite[Appendix (II.21)]{GR}). Performing the substitutions $q\mapsto q^2$, $a=b=q$, $c=q^{1+n}$ and $N=(n-1)/2$
in the above formula, we get the desired identity.
\end{proof}

Like many theorems in \cite{GuoZu}, Theorem \ref{thm:main} has a parametric generalization.
\begin{thm}
Let $n>1$ be an odd integer and let $r\geqslant 1$. Then, modulo
$$
[n^r]\prod_{j=0}^{(n^{r-1}-1)/d}(1-aq^{(2j+1)n})(a-q^{(2j+1)n}),
$$
we have
\begin{align}
&\sum_{k=0}^{(n^r-1)/d}[4k+1]\frac{(aq;q^2)_k (q/a;q^2)_k(q;q^2)_k^2} {(aq^2;q^2)_k (q^2/a;q^2)_k(q^2;q^2)_k^2} \notag\\[5pt]
&\quad\equiv q^{(1-n)/2} [n] \sum_{k=0}^{(n^{r-1}-1)/d}[4k+1]_{q^n}
\frac{(aq^n;q^{2n})_k (q^n/a;q^{2n})_k(q^{n};q^{2n})_k^2} {(aq^{2n};q^{2n})_k (q^{2n}/a;q^{2n})_k(q^{2n};q^{2n})_k^2},  \label{eq:main-a}
\end{align}
where $d=1,2$.
\end{thm}

\begin{proof}
By Lemma~\ref{lem:1} with $n\mapsto n^r$, we see that
the left-hand side of \eqref{eq:main-a} is congruent to $0$
modulo $[n^r]$.  On the other hand, letting $r\mapsto r-1$ and $q\mapsto q^n$ in Lemma~\ref{lem:1}, we conclude that
the summation on the right-hand side of \eqref{eq:main-a} is congruent to $0$ modulo $[n^{r-1}]_{q^n}$.
Furthermore, it is easy to see that, for odd $n$, the $q$-integer $[n]$ is relatively prime to $1+q^k$ for any positive integer $k$.
Hence $[n]$ is also relatively prime to the denominators of the sum on the right-hand side of \eqref{eq:main-a} because of the relation
$$
\frac{(q^n;q^{2n})_k}{(q^{2n};q^{2n})_k}={2k\brack k}_{q^n}\frac{1}{(-q^n;q^n)_k^2},
$$
where ${2k\brack k}_{q^n}=(q^n;q^n)_{2k}/(q^n;q^n)_k^2$ is the central $q$-binomial coefficient.
This means that the right-hand side of \eqref{eq:main-a} is congruent to $0$ modulo
$[n][n^{r-1}]_{q^n}=[n^r]$. Namely, the $q$-congruence \eqref{eq:main-a} is true modulo $[n^r]$.

To prove that it is also true modulo
\begin{align}
\prod_{j=0}^{(n^{r-1}-1)/d}(1-aq^{(2j+1)n})(a-q^{(2j+1)n}), \label{eq:prod}
\end{align}
it suffices to prove that both sides of \eqref{eq:main-a} are equal when $a=q^{-(2j+1)n}$ or $a=q^{(2j+1)n}$ for all $0\leqslant j\leqslant (n^{r-1}-1)/d$, i.e.,
\begin{align}
&\sum_{k=0}^{(n^r-1)/d}[4k+1]\frac{(q^{1-(2j+1)n};q^2)_k (q^{1+(2j+1)n};q^2)_k(q;q^2)_k^2} {(q^{2-(2j+1)n};q^2)_k (q^{2+(2j+1)n};q^2)_k(q^2;q^2)_k^2} \notag\\[5pt]
&\quad=q^{(1-n)/2} [n] \sum_{k=0}^{(n^{r-1}-1)/d}[4k+1]_{q^n}
\frac{(q^{-2jn};q^{2n})_k (q^{(2j+2)n};q^{2n})_k(q^{n};q^{2n})_k^2} {(q^{(1-2j)n};q^{2n})_k (q^{(2j+3)n};q^{2n})_k(q^{2n};q^{2n})_k^2}.  \label{eq:main-a-n}
\end{align}
It is clear that $(n^r-1)/d\geqslant ((2j+1)n-1)/2$ for $0\leqslant j\leqslant (n^{r-1}-1)/d$,
and $(q^{1-(2j+1)n};q^2)_k=0$ for $k>((2j+1)n-1)/2$. By Lemma \ref{lem:2}, the left-hand side of \eqref{eq:main-a-n} is equal to
$q^{(1-(2j+1)n)/2}[(2j+1)n]$. Similarly, the right-hand side of \eqref{eq:main-a-n} is equal to
$$
q^{(1-n)/2} [n]\cdot q^{-jn}[2j+1]_{q^n}=q^{(1-(2j+1)n)/2}[(2j+1)n].
$$
This proves \eqref{eq:main-a-n}. Namely, the $q$-congruence \eqref{eq:main-a} holds modulo \eqref{eq:prod}.
Since $\prod_{j=1}^r\Phi_{n^j}(q)$ and \eqref{eq:prod}  are relatively prime polynomials, we complete the proof of \eqref{eq:main-a}.
\end{proof}

\begin{proof}[Proof of Theorem {\rm\ref{thm:main}}] It is easy to see that the limit of \eqref{eq:prod} as $a\to 1$ has the factor
\begin{align*}
\begin{cases}
\prod_{j=1}^r\Phi_{n^j}(q)^{2n^{r-j}},&\text{if $d=1$,}\\[5pt]
\prod_{j=1}^r\Phi_{n^j}(q)^{n^{r-j}+1},&\text{if $d=2$.}
\end{cases}
\end{align*}
On the other hand, the denominator of the right-hand side of \eqref{eq:main-a} divides that of the left-hand side of \eqref{eq:main-a}.
The factor of the latter related to $a$ is $(aq^2;q^2)_{(n^r-1)/2} (q^2/a;q^2)_{(n^r-1)/2}$, the limit of which as $a\to 1$ only has
the following factor
\begin{align*}
\begin{cases}
\prod_{j=1}^r\Phi_{n^j}(q)^{2n^{r-j}-2},&\text{if $d=1$,}\\[5pt]
\prod_{j=1}^r\Phi_{n^j}(q)^{n^{r-j}-1},&\text{if $d=2$.}
\end{cases}
\end{align*}
related to $\Phi_n(q),\Phi_{n^2}(q),\ldots,\Phi_{n^r}(q)$. Thus, letting $a\to 1$ in \eqref{eq:main-a}, we see that
\eqref{q-c3-1} is true modulo $\prod_{j=1}^r\Phi_{n^j}(q)^3$, one product $\prod_{j=1}^r\Phi_{n^j}(q)$ of which comes from $[n^r]$.

Finally, by \eqref{eq:gw} and \cite[Theorem 4.2]{GuoZu}, we see that
\begin{align*}
\sum_{k=0}^{(n-1)/2}[4k+1]\frac{(q;q^2)_k^4}{(q^2;q^2)_k^4}
\equiv \sum_{k=0}^{n-1}[4k+1]\frac{(q;q^2)_k^4}{(q^2;q^2)_k^4}\equiv 0\pmod{[n]}.
\end{align*}
Replacing $n$ by $n^r$ in the above congruences, we see that the left-hand sides of \eqref{q-c3-1} and \eqref{q-c3-2}
are congruent to $0$ modulo $[n^r]$, while letting $q\mapsto q^n$ and $n\mapsto n^{r-1}$ in the above congruences,
we conclude that the right-hand sides of them are congruent to 0 modulo $[n][n^{r-1}]_{q^n}=[n^r]$.
It follows that \eqref{q-c3-1} and \eqref{q-c3-2} hold modulo $[n^r]$.
The proof of the theorem then follows from the fact that the least common multiple of $\prod_{j=1}^r\Phi_{n^j}(q)^3$ and $[n^r]$
is just $[n^r]\prod_{j=1}^r\Phi_{n^j}(q)^2$.
\end{proof}

\section{Proof of Theorem \ref{thm:main-2}}
Similarly as before, we need the following lemma, which is a special case of \cite[Theorem~4.3]{GuoZu} and can be deduced from \cite[eq.~(4.6)]{Ra93}.
\begin{lem}Let $n$ be a positive odd integer. Then
\begin{align}
\sum_{k=0}^{(n-1)/2} q^{k^2}[6k+1]\frac{(aq;q^2)_k (q/a;q^2)_k (q^2;q^4)_k }{(aq^4;q^4)_k(q^4/a;q^4)_k(q^4;q^4)_k}\equiv 0\pmod{[n]}, \label{eq:6k+1-1} \\[5pt]
\sum_{k=0}^{(n-1)/2} q^{k^2}[6k+1]\frac{(q^{1-n};q^2)_k (q^{1+n};q^2)_k (q^2;q^4)_k }{(q^{4-n};q^4)_k(q^{4+n};q^4)_k(q^4;q^4)_k}  \label{eq:6k+1-2}
=(-q)^{(1-n)/2}[n].
\end{align}
\end{lem}

We also need to establish the following parametric generalization of Theorem \ref{thm:main-2}.
\begin{thm}
Let $n>1$ be an odd integer and let $r\geqslant 1$. Then, modulo
$$
[n^r]\prod_{j=0}^{(n^{r-1}-1)/d}(1-aq^{(2j+1)n})(a-q^{(2j+1)n}),
$$
we have
\begin{align}
&\sum_{k=0}^{(n^r-1)/d} q^{k^2}[6k+1]\frac{(aq;q^2)_k (q/a;q^2)_k (q^2;q^4)_k }{(aq^4;q^4)_k(q^4/a;q^4)_k(q^4;q^4)_k}\notag\\[5pt]
&\quad\equiv (-q)^{(1-n)/2} [n] \sum_{k=0}^{(n^{r-1}-1)/d}
q^{k^2}[6k+1]_{q^2}\frac{(aq^n;q^{2n})_k (q^n/a;q^{2n})_k (q^{2n};q^{4n})_k }{(aq^{4n};q^{4n})_k(q^{4n}/a;q^{4n})_k(q^{4n};q^{4n})_k},  \label{eq:main-2a}
\end{align}
where $d=1,2$.
\end{thm}
\begin{proof}[Sketch of proof]
By \eqref{eq:6k+1-1}, we see that both sides of \eqref{eq:main-2a} are congruent to $0$ modulo $[n^r]=[n][n^{r-1}]_{q^n}$.
Thus, the congruence \eqref{eq:main-2a} holds modulo $[n^r]$.
To prove that \eqref{eq:main-2a} also holds modulo \eqref{eq:prod}, it suffices to show that both sides of \eqref{eq:main-2a}
are identical when  $a=q^{-(2j+1)n}$ or $a=q^{(2j+1)n}$ for all $0\leqslant j\leqslant (n^{r-1}-1)/d$.
We may accomplish this by applying \eqref{eq:6k+1-1} on both sides of \eqref{eq:main-2a} with $a=q^{-(2j+1)n}$.
\end{proof}

\begin{proof}[Proof of Theorem {\rm\ref{thm:main-2}}]
The proof is similar to that of Theorem \ref{thm:main}. Letting $a\to 1$ in \eqref{eq:main-2a},
we conclude that \eqref{q-j3-1} is true modulo $\prod_{j=1}^r\Phi_{n^j}(q)^3$.
On the other hand, by \cite[Theorem 1.3]{Guo-J} or \cite[Theorem 4.3]{GuoZu}, we have
\begin{align*}
\sum_{k=0}^{(n-1)/2} q^{k^2}[6k+1]\frac{(q;q^2)_k^2 (q^2;q^4)_k }{(q^4;q^4)_k^3}
\equiv \sum_{k=0}^{n-1} q^{k^2}[6k+1]\frac{(q;q^2)_k^2 (q^2;q^4)_k }{(q^4;q^4)_k^3}\equiv 0\pmod{[n]}.
\end{align*}
Applying the above congruences, we immediately conclude that both sides of \eqref{q-j3-1} and \eqref{q-j3-2} are all
congruent to $0$ modulo $[n^r]$. This completes the proof.
\end{proof}

\section{Concluding remarks}
We have seen that supercongruences can be proved by establishing suitable $q$-analogues of them.
As mentioned in \cite{GuoZu}, $q$-supercongruences have many virtues
than usual supercongruences. For example, the creative microscoping method used to prove $q$-supercongruences cannot be transformed into a method used to prove usual supercongruences.
Although we have only proved \eqref{eq:c3}, \eqref{eq:j3}, \eqref{eq:cc} and \eqref{eq:jj} modulo $p^{3r}$
instead of expected $p^{4r}$,
no proofs were known before for such congruences even modulo $p^{2r}$.

Note that a complete $q$-analogue of Swisher \cite[(I.3)]{Swisher} has already been given by the author \cite{Guo2};
it has a somewhat different favour.
It is also possible to give $q$-analogues or partial $q$-analogues of some other supercongruences conjectured by Swisher in \cite{Swisher};
we plan to explore them in the near future.

In this context we also need to highlight that certain
supercongruences are related to the coefficients of modular forms.
One famous example is the supercongruence (see Van Hamme \cite[(M.2)]{Hamme} and Kilbourn \cite{Kilbourn})
$$
\sum_{k=0}^{p-1}\frac{(\frac12)_k^4}{k!^4}
\equiv\sum_{k=0}^{(p-1)/2}\frac{(\frac12)_k^4}{k!^4}
\equiv\gamma_p\pmod{p^3}
$$
valid for an odd prime $p$, where $\gamma_p$ stands for the $p$-th coefficient in the $q$-expansion
$$
q(q^2;q^2)_\infty^4(q^4;q^4)_\infty^4=\sum_{n=1}^\infty\gamma_nq^n\quad \text{(of a modular form)}.
$$
Although Swisher \cite{Swisher} did not give a generalization of the (M.2) supercongruence,
Long et al.\ \cite[Section~2.1]{LTYZ17} showed that the proof of the (M.2) supercongruence is equivalent to verifying that
\begin{align}
\sum_{k=0}^{p^{r+1}-1}\frac{(\frac12)_k^4}{k!^4}
\equiv \left(\sum_{k=0}^{p-1}\frac{(\frac12)_k^4}{k!^4}\right) \left(\sum_{k=0}^{p^{r}-1}\frac{(\frac12)_k^4}{k!^4}\right)
\pmod{p^3}  \label{eq:ltyz-1}
\end{align}
holds
for $r=1$ and $2$ (hence for all $r\geqslant 1$).  Recently, the author and Zudilin \cite{GuoZu} suspected that a $q$-analogue of the supercongruences \eqref{eq:ltyz-1}
should exist. Here we formulate such a $q$-analogue.
\begin{conj}
Let $n>1$ be an odd integer and let $r\geqslant 1$. Then, modulo $\Phi_n(q)^3$,
\begin{align}
\sum_{k=0}^{n^{r+1}-1}\frac{(q;q^2)_k^4 }{(q^2;q^2)_k^4 }q^{2k}
\equiv
\left(\sum_{k=0}^{n-1}\frac{(q;q^2)_k^4 }{(q^2;q^2)_k^4 }q^{2k}\right)
\left(\sum_{k=0}^{n^r-1}\frac{(q^{n^2};q^{2n^2})_k^4 }{(q^{2n^2};q^{2n^2})_k^4 }q^{2n^2k}\right).  \label{eq:q-ltyz-1}
\end{align}
\end{conj}

By the Lucas theorem, one sees that $(\frac12)_k/k!={2k\choose k}/4^k\equiv 0\pmod{p}$ for $k$ in the range $(p^s+1)/2\leqslant k\leqslant p^s-1$ where $s=1,2,\ldots.$
Thus, the supercongruence \eqref{eq:ltyz-1} has the following equivalent form:
\begin{align}
\sum_{k=0}^{(p^{r+1}-1)/2}\frac{(\frac12)_k^4}{k!^4}
\equiv \left(\sum_{k=0}^{(p-1)/2}\frac{(\frac12)_k^4}{k!^4}\right) \left(\sum_{k=0}^{(p^{r}-1)/2}\frac{(\frac12)_k^4}{k!^4}\right)
\pmod{p^3}.  \label{eq:ltyz-2}
\end{align}
Likewise, we have the following natural $q$-analogue of \eqref{eq:ltyz-2}.
\begin{conj}
Let $n>1$ be an odd integer and let $r\geqslant 1$. Then, modulo $\Phi_n(q)^3$,
\begin{align}
\sum_{k=0}^{(n^{r+1}-1)/2}\frac{(q;q^2)_k^4 }{(q^2;q^2)_k^4 }q^{2k}
\equiv
\left(\sum_{k=0}^{(n-1)/2}\frac{(q;q^2)_k^4 }{(q^2;q^2)_k^4 }q^{2k}\right)
\left(\sum_{k=0}^{(n^r-1)/2}\frac{(q^{n^2};q^{2n^2})_k^4 }{(q^{2n^2};q^{2n^2})_k^4 }q^{2n^2k}\right).  \label{eq:q-ltyz-2}
\end{align}
\end{conj}

We should mention that the $q$-congruences \eqref{eq:q-ltyz-1} and \eqref{eq:q-ltyz-2} are not equivalent to each other.
This is because the left-hand sides of \eqref{eq:q-ltyz-1} and \eqref{eq:q-ltyz-2} are not congruent to each other even modulo $\Phi_n(q)$.
(Instead, they are congruent to each other modulo $\Phi_{n^{r+1}}(q)^4$.)

Finally, we propose the following partial $q$-analogues of \eqref{eq:ltyz-1} and \eqref{eq:ltyz-2}.
\begin{conj}
Let $n>1$ be an odd integer and let $r\geqslant 1$. Then, modulo $\Phi_n(q)^2$,
\begin{align}
\sum_{k=0}^{(n^{r+1}-1)/d}\frac{(q;q^2)_k^4 }{(q^2;q^2)_k^4 }q^{2k}
\equiv
\left(\sum_{k=0}^{(n-1)/d}\frac{(q;q^2)_k^4 }{(q^2;q^2)_k^4 }q^{2k}\right)
\left(\sum_{k=0}^{(n^r-1)/d}\frac{(q^{n};q^{2n})_k^4 }{(q^{2n};q^{2n})_k^4 }q^{2nk}\right)  \label{eq:q-ltyz-3}
\end{align}
for $d=1,2$.
\end{conj}

So far we did not find any parametric generalizations of \eqref{eq:q-ltyz-1}, \eqref{eq:q-ltyz-2} and \eqref{eq:q-ltyz-3}.
This makes it difficult to use the creative microscoping method here.
In any case, it is the first time when a $q$-version of \eqref{eq:ltyz-1}
is given, though conjectural. We hope that an interested reader will make some progress on it.

\vskip 5mm \noindent{\bf Acknowledgment.} The author thanks Wadim Zudilin for many valuable comments on a previous version of this manuscript.
This work was partially supported by the National Natural Science Foundation of China (grant 11771175).

\end{document}